\RequirePackage{fix-cm}
\documentclass[smallextended]{svjour3}      
\smartqed  

\usepackage{amsmath}
\usepackage{graphicx}
\usepackage{subfigure}
\usepackage{authblk}
\usepackage{amssymb}
\usepackage{algorithm}
\usepackage{algcompatible}

\begin{document}

\title{A Discretely Entropy Conserving/Stable Flux Reconstruction Scheme via Hybridized Split Form Differential Operator and Generalized Entropy Projection}

\titlerunning{A Discretely Entropy Conserving/Stable FR} 
\author{Geng Liang         \and
        Junjie Wang\and
        Hui Xu
}


\institute{Geng Liang, Junjie Wang\at School of Aeronautics and Astronautics, Shanghai Jiao Tong University, Shanghai, 200240, China.  \email{lianggeng@sjtu.edu.cn,  kaga\_mi@sjtu.edu.cn}         
\and
Hui Xu (Corresponding author)\at School of Aeronautics and Astronautics, Shanghai Jiao Tong University, Shanghai, 200240, China. \email{dr.hxu@sjtu.edu.cn}
}

\date{Received: date / Accepted: date}

\maketitle


\abstract{This paper proposes a novel discretely entropy conserving/stable flux reconstruction (FR) scheme for nonlinear conservation laws. The scheme combines a hybridized split-form differential operator with a generalized entropy projection, enabling discrete entropy conservation for arbitrary nodal distributions and arbitrary correction functions without modifying the original quadrature weights. The formulation is built directly upon the classical FR framework and is shown to preserve conservation properties while achieving entropy conservation in the semi-discrete sense. Numerical experiments using the isentropic vortex problem confirm the theoretically predicted accuracy and demonstrate robust entropy behavior with entropy-dissipative interface fluxes.}



\maketitle

\section{Introduction}
The Flux Reconstruction (FR) method, first developed by Huynh, provides a unified framework encompassing major high-order schemes such as Spectral Difference (SD) and Discontinuous Galerkin (DG) methods \cite{FR2007,FR2014}. 
When solving nonlinear PDEs using high-order methods, the numerical scheme must satisfy discretely entropy stability to ensure physically consistent solutions. For simplicity, we directly consider the governing equation in the reference element \(\hat{\Omega}\):
\begin{equation}
    \frac{\partial \boldsymbol{u}}{\partial t}+\sum_{k=1}^d\frac{\partial \boldsymbol{f}_k}{\partial \xi_k}=0,
    \label{conservationlaw}
\end{equation} entropy stability requires \(\partial\eta(\boldsymbol{u})/\partial t+\sum_{k=1}^d\partial\psi_k(\boldsymbol{u})/\partial \xi_k\leq0\), where \(\eta\) denotes the scalar convex entropy function and \(\psi_k\) represents the entropy fluxes with \(\psi_k'=\eta'\boldsymbol{f}_k'\) \cite{ES2003}. A numerical flux satisfies \([\![\boldsymbol{v}]\!]\boldsymbol{f}^*_k=[\![\phi_k]\!]\) to achieve entropy conservation between two states, where \([\![\cdot]\!]\) represents the jump between two values, \(\boldsymbol{v}=\eta'(\boldsymbol{u})\) corresponds to entropy variables, and \(\phi_k=\boldsymbol{v}\boldsymbol{f}_k-\psi_k\) is the entropy potential. However, the derivative terms generated at local solution points in high-order methods inherently violate discretely entropy conservation. The entropy-conserving DG schemes based on the split form achieve discretely entropy conserving property by reformulating the flux function into a symmetric expression compatible with a two-point flux structure \cite{ES2016,SBP2014}. Furthermore, through hybridization of boundary points and entropy projection, discrete entropy conservation can be achieved at Gauss-Legendre nodes that exclude boundary points \cite{ES2018}. However, the split form DG formulation cannot be directly generalized to the FR framework. Current research has achieved the generalization of split-form DG to FR through a weight-adjusted nonlinearly stable flux reconstruction (NSFR) framework by exploiting the equivalence between filtered DG and FR methods \cite{FR2022,FR2025}.

After a careful investigation of the difficulties in extending split-form DG to FR, and motivated by the differential operators for hybridized boundary points, we propose a novel extension approach based on a generalized entropy projection. This formulation preserves the original quadrature weights while extending entropy projection through correction functions. By embedding correction functions more deeply into the hybridized differential operators, the proposed method guarantees discrete entropy conservation for arbitrary nodal distributions and correction functions. 

The present work differs fundamentally from the NSFR approach of \cite{FR2025}. The NSFR method relies on a modified mass matrix to account for the mismatch between solution points and correction functions, whereas our approach preserves the original quadrature weights and instead introduces correction-function-dependent interpolation vectors and solution-point-dependent gradient vectors within a hybridized split-form differential operator. This leads to a mathematically non-equivalent semi-discrete operator that does not require filtered DG equivalence or matrix inversion, and that achieves discrete entropy conservation for arbitrary nodal distributions and correction functions.

\section{Methodology}
\subsection{Classical FR and summation-by-parts property}
In FR semi-discretization, solution values \(\boldsymbol{u}_i(t)\) and the fluxes in the \(k\)-th direction \(\boldsymbol{f}_{i,k}\) are stored at \(N_s\) solution points \(\{\mathcal{X}_i\}_{i=1}^{N_s}\in \hat{\Omega}\), with the polynomial approximation constructed via Lagrange basis functions \(\ell_i(\boldsymbol{\xi})\):  
\begin{equation}
\boldsymbol{u}(\boldsymbol{\xi},t) = \sum_{i=1}^{N_s} \boldsymbol{u}_i(t) \ell_i(\boldsymbol{\xi}), \quad \boldsymbol{f}^D_k(\boldsymbol{\xi},t) = \sum_{i=1}^{N_s} \boldsymbol{f}_{i,k} \ell_i(\boldsymbol{\xi}).
\end{equation}
For \(N_b\) boundary quadrature points \(\{\mathcal{X}^b_m\}_{m=1}^{N_b}\in\partial\hat{\Omega}\) on the reference element's faces, the FR method introduces a set of correction function \(g_{m,k}\) that satisfies condition \cite{FR2016}
\begin{equation}
\sum_{k=1}^dn_k(\boldsymbol{\xi})g_{m,k}(\boldsymbol{\xi})\Bigg|_{\boldsymbol{\xi}=\mathcal{X}_m^b}=1,\quad \sum_{k=1}^dn_k(\boldsymbol{\xi})g_{m,k}(\boldsymbol{\xi})\Bigg|_{\boldsymbol{\xi}=\mathcal{X}^b\backslash\mathcal{X}_m^b}=0,
    \label{correct}
\end{equation}
for \(m=1,\ldots,N_b\), where \(n_{k}\) is the \(k\)-th component of the outer normal vector on the boundary of reference element \(\partial\hat{\Omega}\). The continuous flux is reconstructed as
\begin{equation}
\boldsymbol{f}_k(\boldsymbol{\xi},t)=\boldsymbol{f}^D_k+\boldsymbol{f}^C_k = \sum_{i=1}^{N_s} \boldsymbol{f}_{i,k} \ell_i(\boldsymbol{\xi})+\sum_{m=1}^{N_b}\left(\boldsymbol{f}^*_m\cdot\boldsymbol{n}_m-\boldsymbol{f}^b_m\cdot\boldsymbol{n}_m\right)g_{m,k}(\boldsymbol{\xi}),
\end{equation} 
where the numerical flux \(\boldsymbol{f}^*_{m}\cdot\boldsymbol{n}_m\) is calculated by the local solution \(\boldsymbol{u}^b_m\) at point \(\mathcal{X}_m^b\) and the local solutions of corresponding points in adjacent element. The local flux at the boundary point \(\boldsymbol{f}^b_{m}\) is calculated by the value of the discontinuous flux \(\boldsymbol{f}^D_k\) at point \(\mathcal{X}_m^b\) to ensure conservation. By substituting the obtained continuous flux \(\boldsymbol{f}_k(\boldsymbol{\xi},t)\) and piecewise polynomial approximation solutions \(\boldsymbol{u}(\boldsymbol{\xi},t)\) into the reference element conservation law \eqref{conservationlaw}, the FR discretization can be finally obtained.
For ease of discrete analysis, we represent FR in matrix form and define the solution vector \(\mathbf{U}\) and flux vectors  \(\mathbf{F}_k\) consisting of values \(\boldsymbol{u}_{i}\), \(\boldsymbol{f}_{i,k}\) at the solution points. Furthermore, define \(N_s\times N_s\) differential operator \(\mathbf{D}_k\) in the \(k\)-th direction, \(N_s\) length gradient  vector \(\mathbf{G}_{m}\) of the correction function \(g_{m,k}\), \(N_s\) length interpolation vector \(\mathbf{E}_m\) as follow
\begin{equation}
(\mathbf{D}_k)_{ij}=\frac{\partial\ell_j(\boldsymbol{\xi})}{\partial \xi_k}\Bigg|_{\boldsymbol{\xi}=\mathcal{X}_i},
    (\mathbf{G}_{m})_i=\sum_{k=1}^d\frac{\partial g_{m,k}(\boldsymbol{\xi})}{\partial \xi_k}\Bigg|_{\boldsymbol{\xi}=\mathcal{X}_i}, (\mathbf{E}_{m})_i=\ell_i(\boldsymbol{\xi})\Bigg|_{\boldsymbol{\xi}=\mathcal{X}_m}.
\end{equation}
\begin{theorem}
The classical FR  can be expressed in discrete form as
\begin{equation}
\begin{split}
    &\frac{\partial \mathbf{U}}{\partial t}+\sum_{k=1}^d\mathbf{R}_k+\sum_{m=1}^{N_b}\mathbf{G}_{m}(\boldsymbol{f}^*_{m}\cdot \boldsymbol{n}_m)=0,\\ &\mathbf{R}_k=(\mathbf{D}_k-\sum_{m=1}^{N_b}n_k(\mathcal{X}_m^b)\mathbf{G}_{m}\mathbf{E}_m^T)\mathbf{F}_k
\end{split}
    \label{clfr}
\end{equation}
and maintains conservation properties for arbitrary solution point distributions \(\{\mathcal{X}_i\}_{i=1}^{N_s}\) and correction functions \(g_{m,k}\) satisfying \eqref{correct}.
\end{theorem}
\begin{proof}
    Introduce Gaussian quadrature weights at solution points \(w_i\) and boundary points \(w^b_m\), along with a \(N_s\) length vector \(\mathbf{1}\) of all ones and a \(N_s\times N_s\) diagonal weight matrix \(\mathbf W_{ii}=w_i\), the conservation property of classical FR originates from Gaussian integration as 
\begin{equation}
\begin{split}
(\mathbf{1}^T\mathbf{W}\mathbf{D}_k)_j&=\int_{\hat{\Omega}}\frac{\partial\ell_j(\boldsymbol{\xi})}{\partial \xi_k}dV=\int_{\partial\hat{\Omega}}\ell_j(\boldsymbol{\xi})n_k(\boldsymbol{\xi})dS=\sum_{m=1}^{N_b}n_{k}(\mathcal{X}_m^b)w_m^b(\mathbf{E}_m^T)_j,\\
\mathbf{1}^T\mathbf{W}\mathbf{G}_m&=\int_{\partial\hat{\Omega}}\sum_{k=1}^dg_{m,k}(\boldsymbol{\xi})n_k(\boldsymbol{\xi})dS=w^b_m,\\ &\Rightarrow\mathbf{1}^T\mathbf{W}(\mathbf{D}_k-\sum_{m=1}^{N_b}n_k(\mathcal{X}^b_m)\mathbf{G}_{m}\mathbf{E}_m^T)=0.
\end{split}
\end{equation}
The consistency of element-wise numerical fluxes ensures global conservation properties.
\end{proof}

The FR method achieves equivalence with the nodal DG method only when employing specific correction functions with particular solution point distributions. Specifically, at Gauss-Legendre (GL) nodes, FR reduces to nodal DG when using the \(g_{DG}\) type correction function; At Legendre-Gauss-Lobatto (LGL) nodes, FR reduces to nodal DG with the \(g_{HU}\) type correction function \cite{FR2011,FRN}. Consider that DG method derives from a weighted residual formulation using basis functions as test functions. Thus, degeneration to DG requires the correction function and Lagrange interpolation polynomials at solution points to satisfy relation \(\mathbf{W}\mathbf{G}_m=w_m^b\mathbf{E}_m\).

\begin{definition}
     For a given set of solution points, the nodal DG method guarantees that its differential operator satisfies the following summation-by-parts (SBP) condition based on the divergence theorem \cite{ES2018}
     \begin{equation}
         \mathbf{W}\mathbf{D}_k+(\mathbf{W}\mathbf{D}_k)^T=\sum_{m=1}^{N_b}n_k(\mathcal{X}^b_m)w^b_m\mathbf{E}_m\mathbf{E}_m^T.
     \end{equation}
\end{definition}
\begin{remark}
   Since the SBP property depends solely on the interpolation vectors generated by the solution points and is independent of the correction function gradient vectors, the classical FR scheme \eqref{clfr} can guarantee conservation properties with arbitrary correction functions, but it cannot inherently ensure entropy stability through simple differential operator reconstruction, unlike the DG method.
\end{remark}

\subsection{Entropy-conserving/stable FR via generalized entropy projection}
We introduce two key components to address the SBP property failure caused by solution point-correction function mismatches: 
\begin{itemize}
    \item A generalized interpolation vector \(\mathbf{E}^b_m\) from correction functions which satisfies \(\mathbf{WG}_m=w_m^b\mathbf{E}^b_m\).
    \item A DG-consistent gradient vector \(\mathbf{G}^s_m\) from solution points which satisfies \(\mathbf{WG}^s_m=w_m^b\mathbf{E}_m\).
\end{itemize}
\begin{definition}
For a given set of solution points and correction functions, define the generalized interpolation vector $\mathbf{E}^b_m$ satisfying $\mathbf{W}\mathbf{G}_m = w_m^b \mathbf{E}^b_m$, and the DG-consistent gradient vector $\mathbf{G}^s_m$ satisfying $\mathbf{W}\mathbf{G}^s_m = w_m^b \mathbf{E}_m$. Furthermore, the generalized entropy projection at boundary points is defined as
\begin{equation}
   \boldsymbol{u}_m^b=\boldsymbol{u}(\boldsymbol{v}_m^b),\quad\boldsymbol{v}_m^b=(\mathbf{E}_m^b)^T\mathbf{V},\quad \mathbf{V}_i=\boldsymbol{v}(\boldsymbol{u}_i).
\end{equation}
The final scheme is obtained by reconstructing \(\mathbf{R}_k\) in \eqref{clfr} as
\begin{equation}
\begin{split}
    \mathbf{R}_k^*=&\left[\left(2\mathbf{D}_k-\sum_{m=1}^{N_b}n_k(\mathcal{X}^b_m)\mathbf{G}_m^s\mathbf{E}_m^T\right)\circ \mathbf{F}^\#_k\right]\mathbf{1}\\&+\sum_{m=1}^{N_b} n_k(\mathcal{X}^b_m)\left(\mathbf{G}_m^s\circ \mathbf{F}^{\#,b}_{m,k}-\mathbf{G}_m\left((\mathbf{E}_m\circ \mathbf{F}^{\#,b}_{m,k})^T \mathbf{1}\right)\right)
\end{split}
    \label{ecfr}
\end{equation}
where \(\circ\) denotes the Hadamard product, \((\mathbf{F}_k^\#)_{ij}=\boldsymbol{f}^*_k(\boldsymbol{u}_i,\boldsymbol{u}_j)\) is a \(N_s\times N_s\) matrix computed using two-point entropy-conserving fluxes, and \((\mathbf{F}^{\#,b}_{m,k})_i=\boldsymbol{f}^*_k(\boldsymbol{u}_m^b,\boldsymbol{u}_i)\) is a \(N_s\) length vector obtained by cross-computing entropy-conserving fluxes between solution points and boundary point.
\end{definition}

For analytical convenience, we leverage the SBP property and the mathematical property that \( \mathbf{C}(\mathbf{A} \circ \mathbf{B}) = (\mathbf{C}\mathbf{A}) \circ \mathbf{B} \) (for diagonal matrix \( \mathbf{C} \)) to obtain
\begin{equation}
\begin{split}
\mathbf{W}\mathbf{R}^*_k=&\left[(\mathbf{WD}_k-(\mathbf{WD}_k)^T)\circ \mathbf{F}^\#_k\right]\mathbf{1}\\&+\sum_{m=1}^{N_b} n_k(\mathcal{X}^b_m)w_m^b(\mathbf{E}_m\circ \mathbf{F}^{\#,b}_{m,k}-\mathbf{E}_m^b(\mathbf{E_m}\circ \mathbf{F}^{\#,b}_{m,k})^T \mathbf{1}).
\end{split}
    \label{temp}
\end{equation}
\begin{theorem}
    The reconstruction approach in \eqref{ecfr} preserves the conservation properties of the FR method.
\end{theorem}
\begin{proof}
    Let \(\mathbf{Q}=\mathbf{WD}_k-(\mathbf{WD}_k)^T\), we obtain \(\mathbf{Q}^T=-\mathbf{Q}\). From the consistency of two-point flux and given \(\mathbf{1}^T\mathbf{E}_m^b=1\), thus,
    \begin{equation}
\mathbf{1}^T(\mathbf{Q}\circ\mathbf{F}_k^\#)\mathbf{1}=\frac{1}{2}\mathbf{1}^T\left[(\mathbf{Q}+\mathbf{Q}^T)\circ \mathbf{F}^\#_k\right]\mathbf{1}=0,
\end{equation}
\begin{equation}
\mathbf{1}^T\mathbf{E}_m\circ \mathbf{F}^{\#,b}_{m,k}=\mathbf{1}^T\mathbf{E}_m^b(\mathbf{E_m}\circ \mathbf{F}^{\#,b}_{m,k})^T \mathbf{1}.
    \end{equation}
We finally obtain \(\mathbf{1}^T\mathbf{WR}^*_k=0\), the conservation property is therefore rigorously proved.
\end{proof}
\begin{theorem}
    If the numerical flux at boundary points is entropy-conserving, then the FR scheme \eqref{ecfr} is discretely entropy conserving for arbitrary nodal distributions and correction functions. If the numerical flux is entropy-dissipative, the scheme is discretely entropy stable.
\end{theorem}
\begin{proof}
    Given that \(\mathbf{D}_k\mathbf{1}=0\), the volume integral term consequently satisfies 
    \begin{equation}
    \begin{split}
&\mathbf{V}^T(\mathbf{Q}\circ\mathbf{F}^\#_k)\mathbf{1}\\&=\sum_{i=1}^{N_s}\sum_{j=1}^{N_s}\boldsymbol{v}_i\mathbf{Q}_{ij}\boldsymbol{f}^*_k(\boldsymbol{u}_i,\boldsymbol{u}_j)=\frac{1}{2}\sum_{i=1}^{N_s}\sum_{j=1}^{N_s}\mathbf{Q}_{ij}(\boldsymbol{v}_i-\boldsymbol{v}_j)\boldsymbol{f}^*_k(\boldsymbol{u}_i,\boldsymbol{u}_j)\\&=\frac{1}{2}\sum_{i=1}^{N_s}\sum_{j=1}^{N_s}\mathbf{Q}_{ij}(\phi_{i,k}-\phi_{j,k})=\frac{1}{2}(\boldsymbol{\Phi}_k^T\mathbf{Q1}-\mathbf{1}^T\mathbf{Q}\boldsymbol{\Phi}_k)\\&=-\mathbf{1}^T\mathbf{WD}_k\boldsymbol{\Phi}_k=-\sum_{m=1}^{N_b}n_k(\mathcal{X}^b_m)w^b_m\mathbf{E}_m^T\boldsymbol{\Phi}_k.
    \end{split}
    \end{equation}
Similarly, the hybridized boundary term satisfies
\begin{equation}
\begin{split}
    &\mathbf{V}^T( \mathbf{E}_m\circ \mathbf{F}^{\#,b}_{m,k}-\mathbf{E}_m^b(\mathbf{E_m}\circ \mathbf{F}^{\#,b}_{m,k})^T \mathbf{1})\\&=\sum_{i=1}^{N_s}(\mathbf{E}_m)_i(\boldsymbol{v}_i-\boldsymbol{v}_m^b)\boldsymbol{f}^*_k(\boldsymbol{u}_m^b,\boldsymbol{u}_i)=\sum_{i=1}^{N_s}(\mathbf{E}_m)_i(\phi_{i,k}-\phi_{m,k}^b)\\
    &=\mathbf{E}_m^T\boldsymbol{\Phi}_k-\phi_{m,k}^b\mathbf{E}_m^T\mathbf{1}=\mathbf{E}_m^T\boldsymbol{\Phi}_k-\phi_{m,k}^b.
\end{split}
\end{equation}
Finally obtained \(\mathbf{V}^T\mathbf{WR}^*_k=-\sum_{m=1}^{N_b}n_k(\mathcal{X}^b_m)w^b_m\phi_{m,k}^b\), the local entropy change rate is
\begin{equation}
    \mathbf{V}^T\mathbf{W}\frac{\partial \mathbf{U}}{\partial t}=\sum_{m=1}^{N_b}w^b_m(\boldsymbol{\phi}_{m}^b\cdot\boldsymbol{n}_m)-\sum_{m=1}^{N_b}w_m^b(\boldsymbol{v}_m^b\boldsymbol{f}^*_{m}\cdot \boldsymbol{n}_m).
\end{equation}
If the boundary numerical flux \(\boldsymbol{f}^*_{m}\cdot\boldsymbol{n}_m\) is entropy-conserving, i.e., it satisfies \((\boldsymbol{v}^b_m-\boldsymbol{v}^+)\boldsymbol{f}^*_{m}\cdot\boldsymbol{n}_m=\boldsymbol{\phi}_{m}^b\cdot\boldsymbol{n}_m-\boldsymbol{\phi}^+\cdot\boldsymbol{n}_m\), then the right-hand side vanishes after summation over all faces, and the scheme is discretely entropy conserving. If the boundary numerical flux is entropy-dissipative, then the right-hand side is non-positive, and the scheme is discretely entropy stable. This completes the proof.
\end{proof}
\begin{remark}
     The generalized interpolation process eliminates entropy variation caused by mismatches between correction functions and solution points. Notably, even when using LGL nodes which include boundary points, if the correction function differs from \(g_{HU}\), a generalized interpolation based on the correction function remains necessary. This computes a hybridized boundary value that differs from the solution values at boundary nodes. It should be emphasized that the present generalized interpolation approach is mathematically not equivalent to the formulation in \cite{FR2025}. The weight-adjusted approach modifies the quadrature weights, whereas our method keeps the original weights and instead modifies the interpolation operators, resulting in a different semi-discrete operator.
\end{remark}

Due to the modified interpolation approach, this method cannot directly reduce to classical FR for linear problems. As established in \cite{FR2007}, the accuracy order of FR schemes fundamentally depends on the polynomial degree of the correction functions. However, by analyzing the form of VCJH correction functions, we observe that in elements with \(N+1\) degrees of freedom, the interpolation generated by correction functions achieves only \(N\)-order accuracy \cite{FR2011}. This implies solution degradation, where \(N+1\)-order accuracy requires orthogonality. Consequently, such methods exhibit \(N\)-order accuracy for arbitrary correction functions, and only achieve \(N+1\)-order accuracy when the correction functions precisely align with the solution points. This represents an inherent trade-off for achieving discretely entropy stability in FR, a conclusion substantiated by the results presented in \cite{FR2025}.

\section{Numerical results}
We conduct numerical experiments using the isentropic vortex solution of the Euler equations, and we have verified for various nodal distributions and correction functions that when only entropy-conserving fluxes are employed, the total entropy change remains at machine precision. In the figures presented below, entropy-dissipative fluxes (i.e., entropy-stable fluxes with dissipation) are used.

\begin{figure}
    \centering
    \subfigure[LGL, \(N=3\)]{\includegraphics[width=0.49\linewidth]{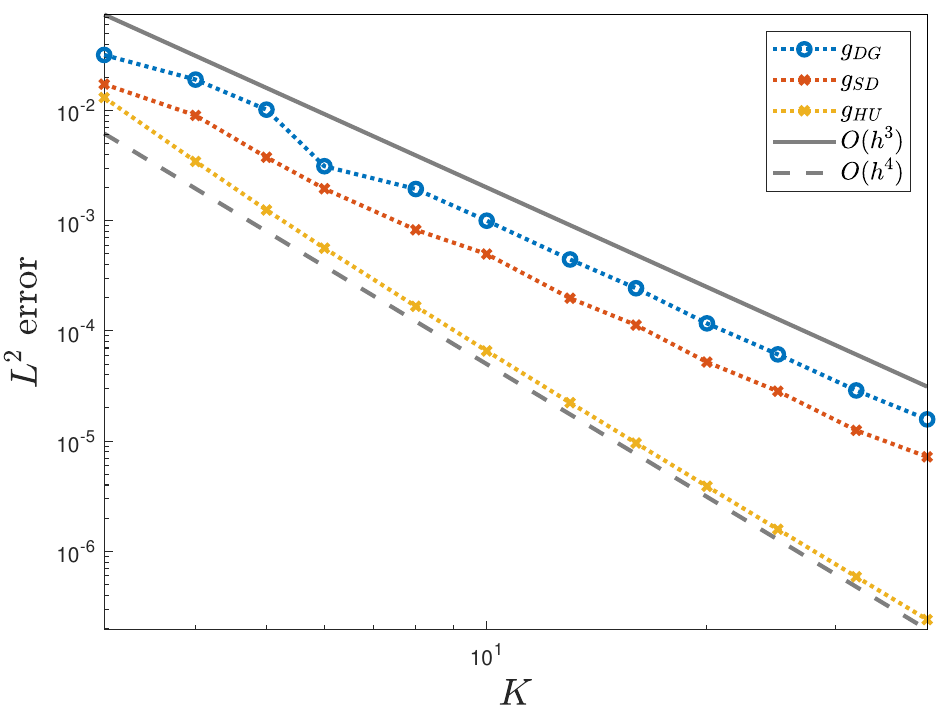}}
    \subfigure[GL, \(N=4\)]{\includegraphics[width=0.49\linewidth]{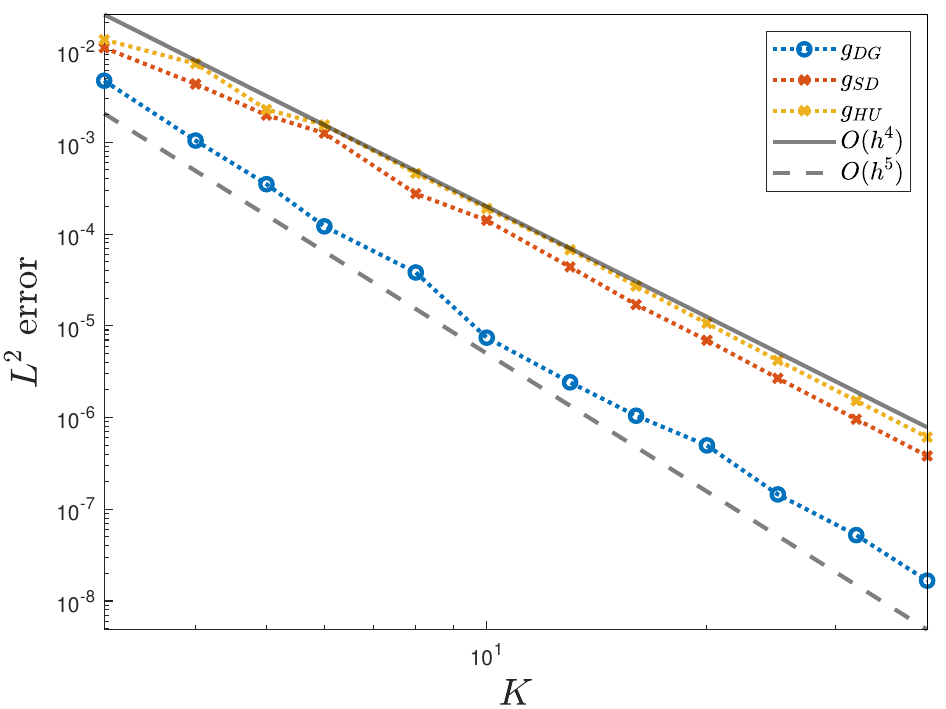}}
    \caption{Accuracy verification: \(L^2\) errors of the proposed discretely entropy stable FR scheme. The observed convergence orders are consistent with the theoretical predictions.}
    \label{fig:accuracy}
\end{figure}

First, we assess the convergence orders. Figure~\ref{fig:accuracy} presents the \(L^2\) errors of the proposed discretely entropy conserving/stable FR scheme on both Legendre-Gauss-Lobatto (LGL) and Gauss-Legendre (GL) nodes. The observed convergence orders are consistent with the theoretical accuracy predictions.

\begin{figure}
    \centering
    \subfigure[\(N=3, K=24\)]{\includegraphics[width=0.49\linewidth]{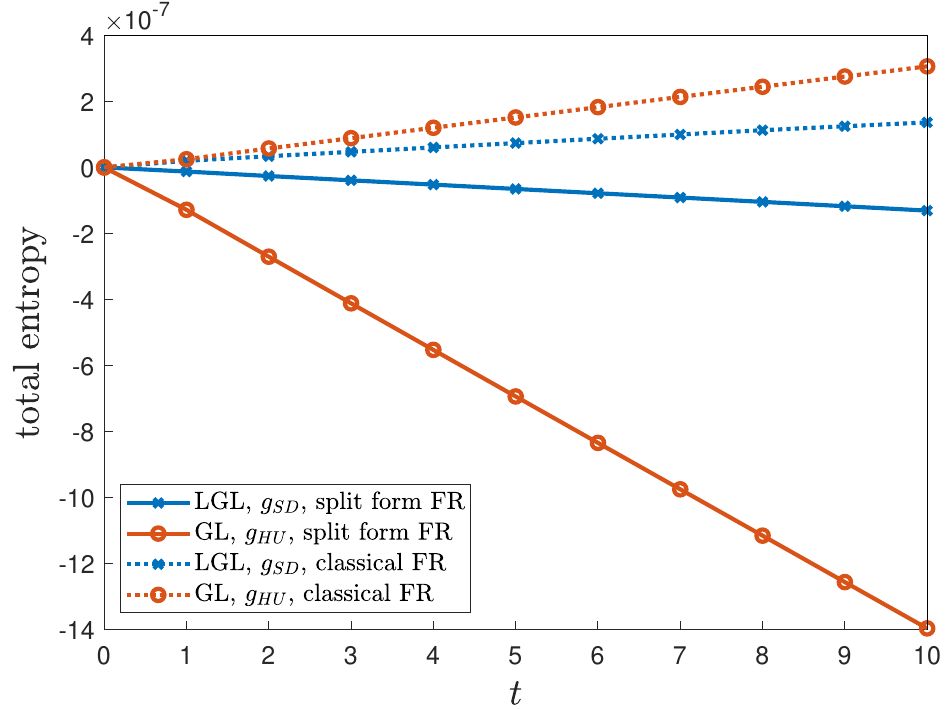}}
    \subfigure[\(N=5, K=16\)]{\includegraphics[width=0.49\linewidth]{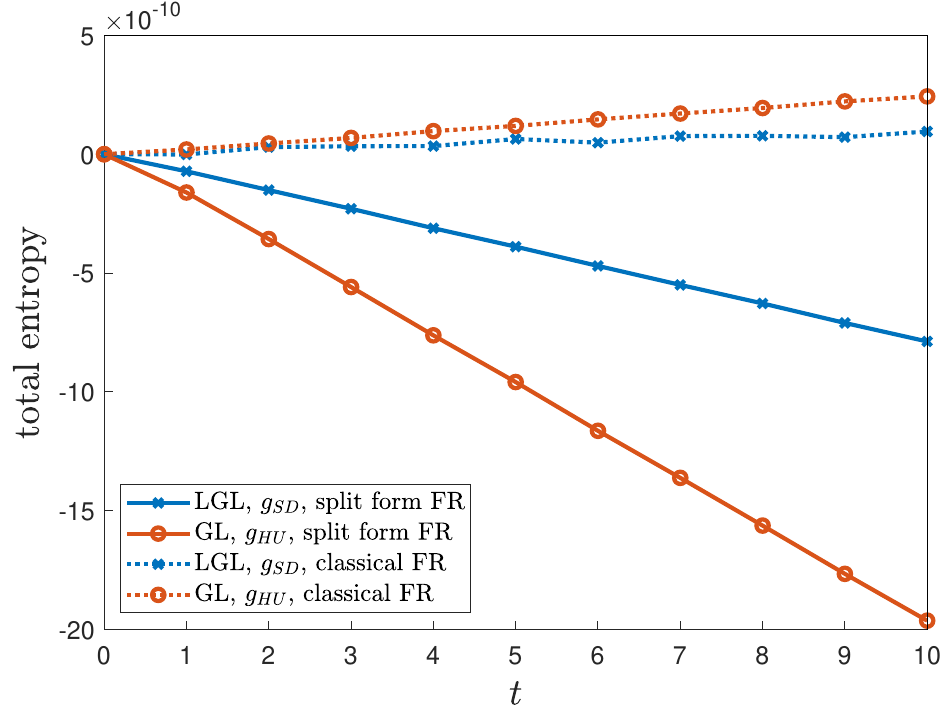}}
    \caption{Total entropy variation for the isentropic vortex problem using entropy-dissipative fluxes. Without the proposed volume-integral treatment, the total entropy would exhibit an anomalous increase.}
    \label{fig:entropy}
\end{figure}

Next, we examine the total entropy variation. Figure~\ref{fig:entropy} shows the total entropy variation when entropy-dissipative fluxes are used. It can be observed that without the proposed volume-integral treatment, the total entropy would still exhibit an anomalous increase, whereas the proposed scheme remains stable.

\section{Conclusions} 
In this work, we proposed a novel entropy-conserving FR extension based on hybridized differential operators and generalized entropy projection, with rigorous proofs of both conservation and discrete entropy stability. The proposed methodology achieves entropy conservation for arbitrary correction functions and nodal distributions. The complete derivation maintains a generalized formulation consistent with classical FR schemes, facilitating application to arbitrary element types.  Crucially, the method is mathematically non-equivalent to the weight-adjusted FR approach of \cite{FR2025}.

\section*{Declarations}

\subsection*{Competing Interests}
The authors have no conflicts of interest to declare that are relevant to the content of this article.

\subsection*{Ethics approval} Not applicable. This work is purely computational and does not involve human participants, animals, or their data.

\subsection*{Data and Code Availability}
The simplified code and data required to reproduce the numerical results reported in this paper are available from the corresponding author upon reasonable request. If the manuscript is considered potentially acceptable, we can prepare a simplified open‑source implementation at the revision stage to ensure complete reproducibility.

\subsection*{Authors' Contributions}
All authors contributed to the study conception and design. The specific contributions according to the CRediT taxonomy are as follows:

\begin{itemize}
    \item \textbf{Geng Liang}: Conceptualization, Investigation, Methodology, Writing -- original draft.
    \item \textbf{Junjie Wang}: Validation.
    \item \textbf{Hui Xu}: Writing -- review \& editing, Supervision, Project administration.
\end{itemize}

\bibliographystyle{spmpsci}
\bibliography{ref}

\end{document}